\documentclass[british, 12pt, reqno]{amsart}
\usepackage[T1]{fontenc}
\usepackage[latin9]{inputenc}
\usepackage{amsmath, amsthm, amssymb, stmaryrd}
\usepackage{mathrsfs}
\usepackage{geometry}
\usepackage{mathtools}

\usepackage{amsmath}
\usepackage{amsthm}
\usepackage{amssymb}
\usepackage{xcolor}
\usepackage{dsfont}

\usepackage{cancel}
\usepackage{soul}
\usepackage{comment}
\usepackage{color}
\usepackage{amscd}
\usepackage{tabularx}
\usepackage{url}
\usepackage{eurosym}
\usepackage{braket}
\usepackage
{hyperref}
\usepackage{enumerate}
\usepackage{graphicx}

\makeatletter
\def\widebreve{\mathpalette\wide@breve}
\def\wide@breve#1#2{\sbox\z@{$#1#2$}%
     \mathop{\vbox{\m@th\ialign{##\crcr
\kern0.08em\brevefill#1{0.8\wd\z@}\crcr\noalign{\nointerlineskip}%
                    $\hss#1#2\hss$\crcr}}}\limits}
\def\brevefill#1#2{$\m@th\sbox\tw@{$#1($}%
  \hss\resizebox{#2}{\wd\tw@}{\rotatebox[origin=c]{90}{\upshape(}}\hss$}
\makeatletter

\makeatletter
\numberwithin{equation}{section}
\numberwithin{figure}{section}
\theoremstyle{plain}
\newtheorem{thm}{\protect\theoremname}[section]
\theoremstyle{plain}

\ifx\proof\undefined
\newenvironment{proof}[1][\protect\proofname]{\par
	\normalfont\topsep6\p@\@plus6\p@\relax
	\trivlist
	\itemindent\parindent
	\item[\hskip\labelsep\scshape #1]\ignorespaces
}{%
	\endtrivlist\@endpefalse
}
\providecommand{\proofname}{Proof}
\fi
\theoremstyle{remark}
\newtheorem{rem}[thm]{\protect\remarkname}
\theoremstyle{plain}
\newtheorem{lem}[thm]{\protect\lemmaname}

\makeatother

\usepackage{babel}
\providecommand{\corollaryname}{Corollary}
\providecommand{\lemmaname}{Lemma}
\providecommand{\remarkname}{Remark}
\providecommand{\theoremname}{Theorem}

\numberwithin{equation}{section}
\numberwithin{figure}{section}
\theoremstyle{plain}

\newtheorem{conj}[thm]{\protect Conjecture}

\def\p{{\textcolor{red} {p}}}
\def\1{{\textcolor{red} {1}}}
\def\d1{{\textcolor{red} {d-1}}}

\def \bN {\mathbb N}

\def \bQ {\mathbb Q}

\def \bR {\mathbb R}

\def \bZ {\mathbb Z}

\def \bl {{\boldsymbol \ell}}

\def \cB {\mathcal B}

\def \cD {\mathcal D}

\def \le {\leqslant}

\def \ge {\geqslant}

\def \d {{\mathrm{d}}}

\def \ds1 {\mathds{1}}

\def \alp {{\alpha}}
\def \bet {{\beta}}
\def \gam {{\gamma}}
\def \del {{\delta}}

\makeatother

\usepackage{babel}
\providecommand{\lemmaname}{Lemma}

\providecommand{\theoremname}{Theorem}
\usepackage{xifthen,ifthen}
\makeatletter
\newcounter{@ToDo}
\newcommand{\todo@helper}[1]{%
	({\color{blue}TODO~\arabic{@ToDo}: {#1\@addpunct{.}}})%
}
\newcommand{\todo}[1]{\stepcounter{@ToDo}%
	\relax\ifmmode\text{\todo@helper{#1}}%
	\else\todo@helper{#1}\fi%
}
\newcounter{@cdo}
\newcommand{\cdo@helper}[1]{%
	({\color{red}CITE~\arabic{@cdo}: {#1\@addpunct{.}}})%
}
\newcommand{\cdo}[1]{\stepcounter{@cdo}%
	\relax\ifmmode\text{\cdo@helper{#1}}%
	\else\cdo@helper{#1}\fi%
}
\newcommand{\mmod}[1]{\,\,\mathrm{mod}\,\,#1}

\def \ds1 {\mathds{1}}

\def \alp {{\alpha}}
\def \bet {{\beta}}
\def \gam {{\gamma}}
\def \del {{\delta}}

\def \epsilon {{\varepsilon}}

\begin{document}

\author{Sam Chow \and Owen Jones}

\address{Mathematics Institute, Zeeman Building, University of Warwick, Coventry CV4 7AL, United Kingdom}
\email{sam.chow@warwick.ac.uk}

\address{Mathematics Institute, Zeeman Building, University of Warwick, Coventry CV4 7AL, United Kingdom}
\email{Owen.S.Jones@warwick.ac.uk}

\title{A note on the $p$-adic inhomogeneous Littlewood conjecture}



\begin{abstract} 
Bugeaud, Haynes and Velani (2011) resolved a metric version of the $p$-adic Littlewood conjecture. We establish a fully inhomogeneous generalisation of their result, as well as a partially inhomogeneous generalisation of a result of Harrap and Haynes (2013) in the more general mixed setting.
\end{abstract}

\maketitle

\section{Introduction}

\subsection{Versions of Littlewood's conjecture}

The Littlewood conjecture is a famous open problem in  diophantine approximation.

\begin{conj}[Littlewood circa 1930]
For all $\alpha,\beta \in \mathbb{R}$,
\[
\liminf_{n\to \infty} n \|n\alpha\| \cdot \|n\beta\|=0.
\]
\end{conj}

As far as we are aware, the following natural inhomogeneous analogue was first stated by Bugeaud~\cite{Bug2014}, albeit as a problem rather than a conjecture. 

\begin{rem}
A very similar question was posed by Cassels~\cite[p. 307]{Cas1997}. That problem was solved by Shapira \cite{Sha2011}, and its $p$-adic analogue was very recently solved \cite{AGVY}.
\end{rem}

\begin{conj}
[Inhomogeneous Littlewood]
Let $\alpha,\beta \in \mathbb{R}$ be such that $1,\alpha$ and $\beta$ are linearly independent over $\mathbb{Q}$. Then, for all $\gam, \del \in \bR$,
\[
\liminf_{n \to \infty} n\|n\alpha-\gam\| \cdot \|n\beta-\del\|=0.
\]
\end{conj}

\begin{rem}
This generalises Littlewood's conjecture, since the latter clearly holds when $1, \alp, \bet$ are linearly dependent over $\bQ$.
\end{rem}

In 2004, de Mathan and Teuli\'e~\cite{dMT} posed the following problem analogous to Littlewood's conjecture, wherein $\|n\beta\|$ is replaced by a $p$-adic absolute value $|n|_p$ of $n$, given by $p^{-k}$ if $p^k \| n$.

\begin{conj}[$p$-adic Littlewood]
Let $\alpha \in \mathbb{R}$ be irrational and $p$ prime. Then 
\[
\liminf_{n \to \infty} n |n|_p \|n\alpha\| = 0.
\]
\end{conj}

Combining the inhomogeneous and $p$-adic aspects gives rise to the following assertion.

\begin{conj}[$p$-adic inhomogeneous Littlewood]
Let $\alpha \in \mathbb{R}$ be irrational and $p$ prime. Then, for all $\gamma\in\mathbb{R}$ and $n_0 \in \mathbb{Z}$, 
\[
\liminf_{n \to \infty}
n|n-n_0|_p \|n\alpha-\gamma\|=0.
    \]
\end{conj}

In 1962, Gallagher~\cite{Gal1962} proved that a stronger form of the Littlewood conjecture holds almost everywhere.

\begin{thm}[Gallagher 1962]
    Let $\psi:\mathbb{N}\to[0,\infty)$ be non-increasing and $k \in \mathbb{N}$.
    If 
    \[
    \sum_{n=1}^{\infty}\psi(n)(\log n)^{k-1}=\infty
    \]
    then, for almost all $(\alp_1,\ldots, \alpha_k) \in \bR^k$, there exist infinitely many $n \in \bN$ such that
\begin{equation}
\label{GallagherConc}
\| n \alp_1 \| \cdots \| n \alp_k \| < \psi(n).
\end{equation}
\end{thm}

Almost half a century later, the natural $p$-adic analogue of this was proven.

\begin{thm} [Bugeaud--Haynes--Velani 2011, $p$-adic Gallagher]
\label{BHVthm}
Let $p_1, \ldots, p_k$ be distinct primes and $\psi: \bN \to [0,\infty)$ be a non-increasing function such that 
\[
\sum_{n=1}^\infty \psi(n) (\log n)^k = \infty.
\]
Then, for almost all $\alp \in \bR$, there exist infinitely many $n \in \bN$ such that
\begin{equation}
\label{MainIneq}
|n|_{p_1} \cdots |n|_{p_k}
\| n \alp \| < \psi(n).
\end{equation}
\end{thm}

Subsequently, Chow and Technau established the natural inhomogeneous generalisation of Gallagher's theorem in \cite{CT2024}.
We establish the following `fully inhomogeneous' generalisation of Theorem \ref{BHVthm}.

\begin{thm}
[Main theorem,  $p$-adic inhomogeneous Gallagher]
\label{MainThm}
Let $p_1, \ldots, p_k$ be distinct primes. Let $n_1, \ldots, n_k \in \bZ$ and $\gam \in \bR$. Let $\psi: \bN \to [0,\infty)$ be a non-increasing function such that 
\[
\sum_{n=1}^\infty \psi(n) (\log n)^k = \infty.
\]
Then, for almost all $\alp \in \bR$, there exist infinitely many $n \in \bN$ such that
\begin{equation}
\label{MainIneq}
|n - n_1|_{p_1} \cdots |n - n_k|_{p_k}
\| n \alp - \gam \| < \psi(n).
\end{equation}
\end{thm}

\begin{rem} If the series converges then, for almost all $\alp \in \bR$, the inequality \eqref{MainIneq} has at most finitely many solutions. This follows from the same routine calculation as in the homogeneous setting \cite{BHV2011}.
\end{rem}

\subsection{Methods}

The Duffin--Schaeffer conjecture \cite{DS1941} was settled by Koukoulopoulos and Maynard \cite{KM2020} in a major breakthrough. Very recently, the natural inhomogeneous analogue was disproved by counterexamples \cite{HMP, HL}. Even more recently, the so-called weak inhomogeneous Duffin--Schaeffer \cite{CT2024} was solved by OpenAI as part of their anthology of over 700 mathematical breakthroughs \cite{OAI}. We now state this, writing $\varphi$ for the Euler totient function. 

\begin{thm} [OpenAI 2026+, weak inhomogeneous Duffin--Schaeffer conjecture]
\label{WeakInhomDuff}
Let $\gamma\in \bR$, and let $\Psi:\bN \to [0,\infty)$ be such that 
\begin{equation}
\label{DivReq}
\sum_{n=1}^\infty \frac{\varphi(n)}{n}\Psi(n)=\infty.
\end{equation}
Then, for almost every $\alpha \in \mathbb{R}$, there exist infinitely many $n\in\bN$ such that 
\[
\|n \alpha - \gamma \| <\Psi(n).
\]
\end{thm}

Some may find this unnatural, owing to the presence of the Euler totient function but the absence of reduced fractions. However, as alluded to above, the claim is false if only reduced fractions are used. Moreover, as we saw in \cite{CT2024}, this apparently weaker version can be useful. We deduce Theorem \ref{MainThm} from Theorem \ref{WeakInhomDuff} by establishing the divergence of the relevant series.

To secure the requisite totient averaging, we apply the AM--GM inequality. This enables us to exploit the product formula for $\varphi$, as in \cite{Cho2018}.

\subsection*{Notation}

The symbol $p$ is reserved for primes.
We adopt the Vinogradov and Bachmann--Landau notations, as we now describe. If $f$ and $g$ are complex-valued functions, then we write $f \ll g$ or $g \gg f$ or $f = O(g)$ if $|f| \le C|g|$ pointwise, for some constant $C>0$.

\subsection*{Organisation}

We prove Theorem \ref{MainThm} in \S \ref{proof}, and include some further discussion in \S \ref{further}.

\subsection*{Funding} 

Owen Jones is supported by the Warwick Mathematics Institute Centre for Doctoral Training, and gratefully acknowledges funding from the Heilbronn Institute for Mathematical Research and the Engineering and Physical Sciences Research Council under the ``Additional Funding Programme for Mathematical Sciences'' (Grant number EP/V521917/1).


\subsection*{AI declaration}

AI has not been used in this work, aside from that we have invoked Theorem \ref{WeakInhomDuff}.

\subsection*{Rights}

For the purpose of open access, the authors have applied a Creative Commons Attribution (CC-BY) licence to any Author Accepted Manuscript version arising from this submission.

\section{Proof of Theorem \ref{MainThm}}
\label{proof}

Throughout this section, we fix $\psi, k$, $n_1,\ldots,n_k$ and $p_1,\ldots,p_k$ as in Theorem~\ref{MainThm}. A quantity will be called small/large if it is small/large in terms of these fixed parameters. Implicit constants will also be allowed to depend on them. We seek to apply Theorem \ref{WeakInhomDuff} with 
\[
\Psi(n)=\frac{\psi(n)}{|n-n_1|_{p_1}\cdots |n-n_k|_{p_k}}.
\]

Our task is to establish \eqref{DivReq}. We do so by working dyadically.
We partition a subset of $(N,2N] \cap \bZ$ according to the shifted $p_j$-adic valuations, and later sum over all sets in the partition. More precisely, we define, for each $\bl= \{ \ell_1,\ldots,\ell_k\} \in \bN^k$,
the set 
\[
B_{\bl}(N):= \{ n \in (N, 2N] \cap \bZ: p_j^{\ell_j} \| (n - n_j) \text{ for all } j \}.
\]
It will suffice to consider $\bl$ from the set
\[
\mathcal{W}(N):= \{ \bl \in \bN^k: p_j^{\ell_j} \le N^{\frac{1}{2k}} \text{ for all } j \}.
\]
The key technical ingredient is that $\varphi(n)/n$ is bounded below on average over each of these sets $B_{\bl}(N)$.

\begin{lem}
Let $N \in \mathbb{N}$ be large. Then,
for each $\bl \in \mathcal{W}(N)$,
\[
\sum_{n \in B_{\bl}(N)} \frac{\varphi(n)}{n} \gg \frac{N}{p_1^{\ell_1}\cdots p_{k}^{\ell_k}}.
    \]
\end{lem}

We take the approach of \cite[Lemma 3.1]{Cho2018}, which concerns totient averaging in Bohr sets.

\begin{proof}
With $P := P(\bl):= p_1^{\ell_1}\cdots p_k^{\ell_k} \le \sqrt N$, let $n_0 = n_0(\bl)$ be an integer such that
\[
n_0 \equiv n_j \mmod p_j^{\ell_j + 1}
\qquad (1 \le j \le k).
\]
This ensures that, for $j=1,2,\ldots,k$,
\[
p_j^{\ell_j} \| (n - n_j)
\quad
\Longleftrightarrow
\quad
p_j^{\ell_j} \| (n - n_0).
\]
Thus, with $q = p_1 \cdots p_k \ll 1$, 
\begin{align*}
\sum_{n \in B_\bl(N)}\frac{\varphi(n)}{n} = \sum_{\substack{ N < n_0 + mP \le 2N
\\(m,q) = 1}} \frac{\varphi(n_0 + mP)}
{n_0 + mP}
= \sum_{\substack{ N < n_0 + mP \le 2N
\\(m,q) = 1}} 
\
\prod_{p \mid (n_0 + mP)} \left(1 - \frac 1 p \right).
\end{align*}
Since $N/P$ is large and $\varphi(q)/q \gg 1$, there are $\gg N/P$ many summands. Hence it suffices, by the AM--GM inequality \cite[Ch. 2]{Ste2004}, to prove that
\begin{equation}
\label{stp}
\prod_{p \le 2N} \left(1 - \frac 1 p \right)^{\alp_p} \gg 1,
\end{equation}
where
\begin{align*}
\alp_p = \frac{P}{N} \# \{ m \in \bZ: (m,q) = 1, \quad p \mid (n_0 + mP) \in (N,2N]
\}.
\end{align*}

Observe that
\[
\prod_{p \le q} \left(1 - \frac 1 p \right)^{\alp_p} \gg 1.
\]
For $p > q$, we have
\[
\alp_p \le \frac{P}{N} \left( \frac{N}{Pp} + 1 \right) \le p^{-1} + N^{-1/2} \ll p^{-1/2}.
\]
Now
\[
-\sum_{p \le 2N} \alp_p \log(1-1/p) \ll \sum_p p^{-3/2} \ll 1,
\]
which delivers \eqref{stp} and completes the proof of the lemma.
\end{proof}

The following calculation yields \eqref{DivReq} and completes the proof of the theorem:
\begin{align*}
\sum_{N < n \le 2N} \frac{\varphi(n) \psi(n)}{n|n-n_1|_{p_1}\cdots |n-n_k|_{p_k}} &\ge \sum_{\bl \in  \mathcal{W}(N)}  \
\sum_{n \in B_{\bl}(N)} \frac{\varphi(n)\psi(n)}{np_1^{-\ell_1} \cdots p_k^{-\ell_k}}\\ &\ge p_1^{\ell_1} \cdots p_k^{\ell_k} \psi(2N) \sum_{\bl \in  \mathcal{W}(N)}  \
\sum_{n \in B_{\bl}(N)}\frac{\varphi(n)}{n}
    \\& \gg p_1^{\ell_1} \cdots p_k^{\ell_k} \psi(2N) (\log N)^k \frac{N}{p_1^{\ell_1}\cdots p_k^{\ell_k}} \\
    &\gg \sum_{2N<n \le 4N} \psi(n) (\log n)^k.
\end{align*}

\section{Further discussion}
\label{further}

\subsection{An alternate proof}

The weak inhomogeneous Duffin--Schaeffer conjecture is a deep result, whose proof we have barely begun to digest. The second author was working on an independent proof of our main theorem, and hopes to publish that upon completion.

\subsection{Another generalisation}

The mixed Littlewood conjecture generalises the $p$-adic Littlewood conjecture (for one prime) by replacing $(p^\ell)_{\ell=1}^\infty$ with $(d_\ell)_{\ell=1}^\infty$ for some increasing sequence of positive integers $d_1  \mid d_2 \mid \ldots$.  Denoting the sequence $(d_\ell)_{\ell=1}^\infty$ by $\cD$, we define the $\cD$-adic pseudo-absolute value $|\cdot|_\cD$ by
\[
|n|_\cD = \min \{ d_\ell^{-1}: d_\ell \mid n \}
\qquad (n \in \bN).
\]

\begin{conj}
[Mixed inhomogeneous Littlewood]
Let $\alp \in \bR$ be irrational, and let $\cD = (d_\ell)_{\ell = 1}^\infty$ be an increasing sequence of positive integers such that $d_\ell \mid d_{\ell + 1}$ for all $\ell$. Then, for all $n_0 \in \bZ$ and all $\gam \in \bR$,
\[
\liminf_{n \to \infty} n 
| n - n_0|_\cD 
\| n \alp - \gam \| = 0.
\]
\end{conj}

This motivates the following hypothetical generalisation of Theorem \ref{MainThm}.
See also \cite[Theorem 3]{HH2013}, which sheds some light on the homogeneous version.

\begin{conj}
[Mixed inhomogeneous Gallagher]
\label{MIG}
Let $\cD =(d_\ell)_{\ell = 1}^\infty$ be an increasing  sequence of positive integers such that $d_\ell \mid d_{\ell+1}$ for all $\ell$. Let $n_0 \in \bZ$ and $\gam \in \bR$, and let $\psi: \bN \to [0, \infty)$ be a non-increasing function such that
\[
\sum_{n=1}^\infty
\psi(n) \max \{ k: d_k \le n \}
= \infty.
\]
Then, for almost all $\alp \in \bR$, there exist infinitely many $n \in \bN$ such that
\[
|n - n_0|_\cD \| n \alp - \gam \| < \psi(n).
\]
\end{conj}

Our approach does not seem to handle this mixed setting in general since, for example,
\[
\frac{\varphi(\ell!)}{\ell!} \to 0 \qquad (\ell \to \infty).
\]
We note, however, that a `partially inhomogeneous' generalisation of \cite[Theorem 3]{HH2013} can be inferred from Theorem \ref{WeakInhomDuff}.

\begin{thm}
Let $\cD = (d_\ell)_{\ell = 1}^\infty$ be an increasing sequence of positive integers such that $d_\ell \mid d_{\ell+1}$ for all $\ell$ and, for all $L \in \bN$,
\[
\sum_{\ell \le L} \frac{\varphi(d_\ell)}{d_\ell} \gg L.
\]
Let $\gam \in \bR$, and let $\psi: \bN \to [0, \infty)$ be a non-increasing function such that 
\[
\sum_{n=1}^{\infty}\psi(n) \max \{ k: d_k \le n \} = \infty.
\]
Then, for almost all $\alp \in \bR$, there exist infinitely many $n \in \bN$ such that
\[
|n |_\cD \| n \alp - \gam \| < \psi(n).
\]
\end{thm}

\begin{proof}
Let $N \in \bN$ be large in terms of the sequence $\cD$, and let $L \in \bN$ be maximal such that $d_L \le 2N$. Then $L$ is also large.
Define, for $\ell \in \bN$,
\[
\cB_\ell(N) = \{ n \in \bZ: N < n \le 2N, \quad
d_\ell \mid n, \quad d_{\ell + 1} \nmid n \}.
\]
Note that if $\ell \le L/2$ then
\[
\frac{2N}{d_{\ell+1}} \ge
\frac{d_L}{d_{\ell+1}} \ge 2^{L - \ell - 1} \ge 2^{(L-2)/2} \ge 100.
\]
We claim that, for each $\ell \le L/2$,
\[
\sum_{n \in \cB_\ell (N)} \frac{\varphi(n)}{n} \gg
\frac{\varphi(d_\ell) N}{d_\ell^2}.
\]

To prove the claim, we use the fact that $\varphi(xy) \ge \varphi(x) \varphi(y)$ for all $x, y \in \bN$. Writing $r_\ell = d_{\ell+1}/d_\ell$,
\begin{align*}
\sum_{n \in \cB_\ell (N)} \frac{\varphi(n)}{n} \ge \frac{\varphi(d_\ell)}{d_\ell}
\
\sum_{\substack{N/d_\ell < m \le 2N/d_\ell \\ r_\ell \nmid m}} \frac{\varphi(m)}{m}.
\end{align*}
There are $\gg N/d_\ell$ many summands so, by the AM--GM inequality, it suffices for the claim to prove that
\begin{equation}
\label{stp2}
\prod_{p} \left(1 - \frac1p \right)^{\alp_p} \gg 1,
\end{equation}
where now we define
\[
\alp_p = \frac{d_\ell}{N}
\# \{ m \in \bZ: r_{\ell} \nmid m,
\ p \mid m \in (N/d_\ell, 2N/d_\ell] \}.
\]
Since 
\[
\alp_p \le \frac{d_\ell}{N}
\# \{ m \in \bZ:
p \mid m \in [1, 2N/d_\ell] \}
\le 2/p
\]
for all primes $p$, we have
\[
- \sum_p \alp_p \log(1- 1/p) \ll \sum_p p^{-2} \ll 1,
\]
which furnishes \eqref{stp2} and with it the claim.

With the claim at hand, we compute that
\begin{align}
\notag
\sum_{N < n \le 2N} \frac{\varphi(n) \psi(n)}{n | n |_\cD} &\ge \sum_{\ell \le L/2} \
d_\ell \psi(2N)
\sum_{n \in \cB_\ell(N)} 
\frac{\varphi(n)}{n} \\
\label{KeyAvg}
&\gg N \psi(2N) \sum_{\ell \le L/2} \frac{\varphi(d_\ell)}{d_\ell} \gg N \psi(2N)L.
\end{align}
By the Cauchy condensation test,
\[
\sum_{m=1}^\infty 2^m \psi(2^m)  \max \{ k: d_k \le 2^m \} = \infty.
\]
In particular, by summing \eqref{KeyAvg} over large $N = 2^{m-1}$, we find
\[
\sum_{n=1}^\infty \frac{\varphi(n) \psi(n)}{n | n |_\cD} = \infty.
\]
The desired conclusion now follows upon applying Theorem \ref{WeakInhomDuff}.
\end{proof}

\end{document}